\documentclass{amsart}

\usepackage[dvipdfm]{graphicx,color}

\newtheorem{theorem}{Theorem}[section]
\newtheorem{lemma}[theorem]{Lemma}
\newtheorem{corollary}[theorem]{Corollary}
\newtheorem{proposition}[theorem]{Proposition}

\theoremstyle{definition}
\newtheorem{definition}[theorem]{Definition}
\newtheorem{example}[theorem]{Example}

\theoremstyle{remark}
\newtheorem{remark}[theorem]{Remark}

\numberwithin{equation}{section}

\newcommand{\trace}{\mathop{\mathrm{trace}}\nolimits}
\newcommand{\red}[1]{\textcolor{black}{#1}}
\newcommand{\blue}[1]{\textcolor{black}{#1}}

\def\tr#1{\mathord{\mathopen{{\vphantom{#1}}^t}#1}} 
\def\rank{\mathop{\mathrm{rank}}\nolimits}
\def\ord{\mathop{\mathrm{ord}}\nolimits}

\def\C{{\mathbb{C}}}
\def\R{{\mathbb{R}}}
\def\H{{\mathcal{H}}}
\def\Z{{\mathbb{Z}}}
\def\Pi{{\mathbb{P}}}
\def\Si{{\mathbb{S}}}
\def\E{{\mathcal{E}}}

\begin{document}

\title[The hyperbolic Gauss maps]{Value distribution of the hyperbolic Gauss maps for 
flat fronts in hyperbolic three-space}

\author[Y.~Kawakami]{Yu Kawakami}
\address{Faculty of Mathematics, Kyushu university, 
744, Motooka, Nishiku, Fukuoka-city,
819-0395,Japan
}
\email{kawakami@math.kyushu-u.ac.jp}
\thanks{Partly supported by Global COE program (Kyushu university) 
``Education and Research Hub for Mathematics-for-Industry'' and 
the Grants-in-Aid for Young Scientists (B) No. 21740053, from the Japan Society for the Promotion of Science.}



\date{}



\begin{abstract}
We give an effective estimate for the totally ramified value number of the hyperbolic Gauss maps of 
complete flat fronts in the hyperbolic three-space. As a corollary, we give the upper bound \red{for} the number of 
exceptional values of them \red{in} some topological cases. 
Moreover, we obtain some new examples for this class. 
\end{abstract}

\maketitle
\section*{Introduction}
The study of flat surfaces in the hyperbolic $3$-space ${\H}^{3}$ has made \red{significant} advance\red{s} 
in the last decade. 
Indeed, G\'alvez, Mart\'inez and Mil\'an \cite{GMM} established a Weierstrass-type representation formula 
for such surfaces. Moreover, Kokubu, Umehara and Yamada (\cite{KUY1}, \cite{KUY2}) investigated global 
properties of flat surfaces in ${\H^{3}}$ with certain kind\red{s} of singularities, called flat fronts 
(\red{for a} precise definition, see Section $1$ of this paper).  
In particular, they gave a representation formula \red{for} constructing 
a flat front from a given pair of hyperbolic Gauss maps and an Osserman-type inequality for complete 
(in the sense of \cite{KUY2}, see also Section $1$ of this paper) flat fronts.  
More recently, Kokubu, Rossman, Saji, Umehara and Yamada \cite{KRSUY} gave criteria for a singular point 
on a flat front in ${\H}^{3}$ be a cuspidal edge or swallowtail and proved the generically flat fronts 
in ${\H}^{3}$ admit only cuspidal edges and swallowtails. 
Moreover, Roitman \cite{Ro} and Kokubu, Rossman, Umehara and Yamada \cite{KRUY1} obtained interesting results 
on flat surfaces or (p-)fronts in ${\H}^{3}$ and their caustics. 
Furthermore, Kokubu, Rossman, Umehara and Yamada \cite{KRUY2} also investigate the asymptotic behavior of ends 
of flat fronts in ${\H}^{3}$. 
However, \red{up to now}, we have not seen \red{a} study of \red{the} value distribution of t
he hyperbolic Gauss maps for complete flat fronts in ${\H}^{3}$.  

On the other hand, we have recently obtained some results on \red{the} value distribution of 
the Gauss map of complete minimal surfaces in Euclidean $3$-space ${\R}^{3}$ and 
the hyperbolic Gauss map of complete constant mean curvature one (CMC-$1$, for short) surfaces in ${\H}^{3}$. 
For instance, we \cite{Ka1} found algebraic minimal surfaces in ${\R}^{3}$ 
with totally ramified value number of the Gauss map equal\red{ing} $2.5$ 
(By an algebraic minimal surface, we mean a complete minimal surface with finite total curvature). 
Moreover, the author, Kobayashi and Miyaoka \cite{KKM} gave an effective estimate for the number of 
exceptional values and the totally ramified value number of the Gauss map 
\red{of} a wider class of complete minimal surfaces that includes algebraic minimal surfaces (this class is called 
``pseudo-algebraic''). \red{In \cite{KKM}}, we also provided new proofs of the Fujimoto \cite{Fu} and Osserman theorems 
(\cite{Os1}, \cite{Os2}) for this class and revealed the geometric meaning behind \red{them}. 
Furthermore, we \cite{Ka3} gave the definition of ``pseudo-algebraic'' and ``algebraic'' CMC-$1$ surfaces in ${\H}^{3}$\red{,} 
and \red{also} such an estimate for the hyperbolic Gauss map of these surfaces. 
These estimates correspond to the defect relation in Nevanlinna theory (\cite{JR}, \cite{Ko}, \cite{NO}, and \cite{Ru}). 

The purpose of this paper is to study \red{the} value distribution of the hyperbolic Gauss maps of flat fronts 
in ${\H}^{3}$. In Section $1$, we recall the definition and some fundamental properties of flat fronts 
in ${\H}^{3}$. In particular, we review a construction of complete flat fronts via a given pair of hyperbolic 
Gauss maps\red{,} and \red{consider a} Osserman-type inequality for this class. In Section $2$, we give an estimate for the totally ramified 
value number of the hyperbolic Gauss maps of complete flat fronts in ${\H}^{3}$ (Theorem \ref{thm-ramification}). 
This estimate is effective in the sense that the lower bound which we obtain is described in terms of geometric 
invariants. We remark that this estimate is similar to the ramification estimate for the Gauss maps of 
complete minimal surfaces in Euclidean $4$-space ${\R}^{4}$ (\cite{Fu}, \cite{HO}, \cite{Ka2}). 
Moreover, as a corollary of this estimate, we give the upper bounds \red{for} the number of exceptional values 
of them \red{in} some topological cases. 
Furthermore, we consider the Fujimoto-Hoffman-Osserman problem for this class, that is, 
the problem of finding the ``common'' maximal number of exceptional values of 
the hyperbolic Gauss maps for complete flat fronts in ${\H}^{3}$. 
In Section $3$, we investigate examples of complete flat fronts in ${\H}^{3}$ from 
the \red{viewpoint} of value distribution of the hyperbolic Gauss maps and 
give some new examples of complete flat fronts in ${\H}^{3}$. 

The author would like to thank Professors Ryoichi Kobayashi, Masatoshi Kokubu, Pablo Mira, Reiko Miyaoka, 
Junjiro Noguchi, Wayne Rossman, Masaaki Umehara and Kotaro Yamada for their useful advice. 

\section{Preliminaries}
In this section, we  briefly recall definition\red{s} and some basic facts \red{about} flat fronts in ${\H}^{3}$. 
For details, we refer the reader to \cite{GMM}, \cite{KRUY1}, \cite{KRUY2}, \cite{KUY1} and \cite{KUY2}. 

Let ${\R}^{4}_{1}$ be the Lorentz-Minkowski $4$-space with the Lorentz metric 
\begin{equation}\label{Lmetric}
\langle (x_{0}, x_{1}, x_{2}, x_{3}), (y_{0}, y_{1}, y_{2}, y_{3}) \rangle = -x_{0}y_{0}+x_{1}y_{1}+x_{2}y_{2}+x_{3}y_{3}\;.
\end{equation}
Then the hyperbolic 3-space is \blue{given by}
\begin{equation}\label{hyperbolic-space}
{\H}^{3}=\{(x_{0}, x_{1}, x_{2}, x_{3})\in {\R}^{4}_{1} \, | \, -(x_{0})^{2}+(x_{1})^{2}+(x_{2})^{2}+(x_{3})^{2}=-1, x_{0}>0 \}
\end{equation}
with the induced metric from ${\R}^{4}_{1}$, which is a simply connected Riemannian $3$-manifold with constant 
sectional curvature $-1$. We identify ${\R}^{4}_{1}$ with the set of $2\times 2$ Hermitian matrices 
Herm($2$)$=\{X^{\ast}=X\}$ $(X^{\ast}:=\tr{\overline{X}}\,)$ by
\begin{equation}\label{Hermite}
(x_{0}, x_{1}, x_{2}, x_{3}) \longleftrightarrow \left(
\begin{array}{cc}
x_{0}+x_{3} & x_{1}+ix_{2} \\
x_{1}-ix_{2} & x_{0}-x_{3}
\end{array}
\right)\, ,
\end{equation}
where $i=\sqrt{-1}$ . \red{With} this identification, $\H^{3}$ is represented as 
\begin{equation}\label{hyperbolicspace}
\H^{3}=\{aa^{\ast}\,|\, a\in SL(2,\C)\}
\end{equation}
with the metric 
\[
\langle X, Y \rangle = -\frac{1}{2}\trace{(X\widetilde{Y})}, \quad   \langle X, X \rangle =-\det(X)\, ,
\]
where $\widetilde{Y}$ is the cofactor matrix of $Y$. The complex Lie group $ PSL(2,\C):= SL(2,\C)/\{\pm \text{id} \}$ 
acts isometrically on 
$\H^{3}$ by 
\begin{equation}\label{action}
\H^{3} \ni X \longmapsto aXa^{\ast}\, , 
\end{equation}
where $a\in PSL(2,\C)$. 

Let $M$ be an oriented $2$-manifold. A smooth map $f\colon M\to {\H}^{3}$ is called a {\it front} 
if there exists a Legendrian immersion 
\[
L_{f}\colon M \to T_{1}^{\ast}{\H}^{3}
\]
into the unit cotangent bundle of ${\H}^{3}$ whose projection is $f$. 
Identifying $T_{1}^{\ast}{\H}^{3}$ with the unit tangent bundle $T_{1}{\H}^{3}$, 
we can write $L_{f}=(f, \nu)$, where $\nu (p)$ is a unit vector in $T_{f(p)}{\H}^{3}$ such that 
$\langle df(p), \nu (p) \rangle= 0$ for each $p\in M$. We call ${\nu}$ a {\it unit normal vector field} of the front $f$. 
\red{A} front may have singular points (i.e., points of $\rank{(df)}<2$). 
A point which is not singular is said to be {\it regular}, where the first fundamental form is positive definite. 

The {\it parallel front} $f_{t}$ of a front $f$ at distance $t$ is given by $f_{t}(p)=\text{Exp}_{f(p)}(t\nu (p))$, 
where ``$\text{Exp}$'' denotes the exponential map of ${\H}^{3}$. 
In the model for ${\H}^{3}$ as in \eqref{hyperbolic-space}, we can write 
\begin{equation}\label{parallel}
f_{t}=(\cosh{t})f+(\sinh{t})\nu, \quad {\nu}_{t}=(\cosh{t})\nu +(\sinh{t})f\,,
\end{equation}
where ${\nu}_{t}$ is the unit normal vector field of $f_{t}$. 

Based on the fact that any parallel surface of a flat surface is also flat at regular points, 
we define flat fronts as follows\red{:} A front $f\colon M\to {\H}^{3}$ is called a {\it flat front} 
if, for each $p\in M$, there exists a real number $t\in \R$ such that the parallel front $f_{t}$ is a 
flat immersion at $p$. By definition, $\{f_{t}\}$ forms a family of flat fronts. 
We remark that an equivalent definition of flat front\red{s} is that the Gaussian curvature 
of $f$ vanishes at all regular points. However, there exists \red{a} case where this definition is not suitable. 
For detail\red{s}, see \cite[Remark 2.2]{KUY2}. 

We assume that $f$ is flat. Then there exists a (unique) complex structure on $M$ and a holomorphic Legendrian immersion 
\begin{equation}\label{Legen-lift}
{\E}_{f}\colon \widetilde{M}\to SL(2,\C)
\end{equation}
such that $f$ and $L_{f}$ are projections of ${\E}_{f}$, where $\widetilde{M}$ is the universal covering of $M$. 
Here, holomorphic Legendrian map means that ${\E}^{-1}_{f}d{\E}_{f}$ is off-diagonal (see \cite{GMM}, \cite{KUY1}, 
\cite{KUY2}). 
We call ${\E}_{f}$ the {\it holomorphic Legendrian lift} of $f$. 
The map $f$ and its unit normal vector field $\nu$ are 
\begin{equation}\label{Legen-map-vec}
f={\E}_{f}{\E}^{\ast}_{f}, \quad \nu = {\E}_{f}e_{3}{\E}^{\ast}_{f}, \quad e_{3}=\left(
\begin{array}{cc}
1 & 0 \\
0 & -1
\end{array}
\right)\,.
\end{equation}

If we set 
\begin{equation}\label{Legen-form}
{\E}^{-1}_{f}d{\E}_{f}=\left(
\begin{array}{cc}
0      & \theta \\
\omega & 0
\end{array}
\right)\, ,
\end{equation}
the first and second fundamental forms $ds^{2}=\langle df, df \rangle$ and 
$dh^{2}=-\langle df, d\nu \rangle$ are given by
\begin{eqnarray}\label{Legen-form2}
ds^{2}&=&|\omega+\bar{\theta}|^{2}=Q+\bar{Q}+(|\omega|^{2}+|\theta|^{2}), \quad Q=\omega\theta \nonumber \\
dh^{2}&=&|\theta|^{2}-|\omega|^{2}
\end{eqnarray}
for holomorphic $1$-forms $\omega$ and $\theta$ on $\widetilde{M}$, with $|\omega|^{2}$ and $|\theta|^{2}$ 
\red{well defined} on $M$ itself. 
We call $\omega$ and $\theta$ the {\it canonical forms} of $f$. The holomorphic $2$-differential $Q$ appearing 
in the $(2,0)$-part of $ds^{2}$ is defined on $M$, and is called the {\it Hopf differential} of $f$. 
By definition, the umbilic points of $f$ equal the zeros of $Q$. Defining a meromorphic function on $\widetilde{M}$ by 
\begin{equation}\label{sing-rho}
\rho=\dfrac{\theta}{\omega}\,,
\end{equation}
then $|\rho|\colon M\to [0, +\infty]$ is well-defined on $M$, and $p\in M$ is a singular point if and only if 
$|\rho(p)|=1$. 

Note that the $(1, 1)$-part of the first fundamental form 
\begin{equation}\label{eq-Sasakian}
ds^{2}_{1,1}=|\omega|^{2}+|\theta|^{2}
\end{equation}
is positive definite on $M$ because it is the pull-back of the canonical Hermitian metric of $SL(2, \C)$. 
Moreover, $2ds^{2}_{1,1}$ coincides with the pull-back of the Sasakian metric on $T^{\ast}_{1}{\H}^{3}$ 
by the Legendrian lift $L_{f}$ of $f$ (which is the sum of the first and third fundamental forms in this case, 
see \cite[Section 2]{KUY2} for detail\red{s}). The complex structure on $M$ is compatible with the conformal 
metric $ds^{2}_{1,1}$. Note that any flat front is orientable (\cite[Theorem B]{KRUY1}). 
In this paper, for each flat front $f\colon M\to {\H}^{3}$, 
we always regard $M$ as a Riemann surface with this complex structure. 

The two {\it hyperbolic Gauss maps} are defined by 
\begin{equation}\label{def-Gauss-map}
G=\dfrac{E_{11}}{E_{21}},\quad G_{\ast}=\dfrac{E_{12}}{E_{22}},\quad  \text{where}\quad  {\E}_{f}=(E_{ij})\,.
\end{equation}
By identifying the ideal boundary ${\Si}^{2}_{\infty}$ of ${\H}^{3}$ with the Riemann sphere ${\C}\cup\{\infty\}$, 
the geometric meaning of $G$ and $G_{\ast}$ is given as follows (\cite[Appendix A]{KRUY2}, \cite{Ro}): The hyperbolic 
Gauss maps $G$ and $G_{\ast}$ send each point $p\in M$ to the \blue{terminal} points $G(p)$ and $G_{\ast}(p)$ 
in ${\Si}^{2}_{\infty}$ the two oppositely-oriented normal geodesics of ${\H}^{3}$ that \blue{starting} $f(p)$. 
In particular, $G$ and $G_{\ast}$ are meromorphic functions on $M$ and parallel fronts have the same hyperbolic 
Gauss maps. The transformation ${\E}_{f}\mapsto a{\E}_{f}$ by $a=(a_{ij})_{i,j=1,2}\in SL(2,\C)$ induces the rigid 
motion $f\mapsto afa^{\ast}$ as in \eqref{action} and the hyperbolic Gauss maps $G$ and $G_{\ast}$ change by 
the M\"obius transformation: 
\begin{equation}\label{equ-Mobius}
G\mapsto a\star G=\dfrac{a_{11}G+a_{12}}{a_{21}G+a_{22}}, \quad G_{\ast}\mapsto a\star G_{\ast}=\dfrac{a_{11}G_{\ast}+a_{12}}{a_{21}G_{\ast}+a_{22}}\,.
\end{equation}

Here, we remark \red{on} the interchangeability of the canonical forms and the hyperbolic Gauss maps. 
The canonical forms $(\omega, \theta)$ have 
the $U(1)$-ambiguity $(\omega, \theta)\mapsto (e^{is}\omega, e^{-is}\theta)\,(s\in \R)\red{,}$ which corresponds to 
\begin{equation}\label{equ-U(1)}
{\E}_{f}\longmapsto {\E}_{f}\left(
\begin{array}{cc}
e^{is/2} & 0 \\
0 & e^{-is/2}
\end{array}
\right)\red{.}\, 
\end{equation}
For a second ambiguity, defining the {\it dual} ${\E}_{f}^{\natural}$ of ${\E}_{f}$ by 
\[
{\E}_{f}^{\natural}={\E}_{f}\left(
\begin{array}{cc}
0 & i \\
i & 0
\end{array}
\right),
\] 
then ${\E}_{f}^{\natural}$ is also Legendrian with $f={\E}_{f}^{\natural}{{\E}_{f}^{\natural}}^{\ast}$. 
The hyperbolic Gauss maps $G^{\natural}$, $G_{\ast}^{\natural}$ and canonical forms ${\omega}^{\natural}$, 
${\theta}^{\natural}$ of ${\E}_{f}^{\natural}$ satisfy 
\[
G^{\natural}=G_{\ast}, \quad G_{\ast}^{\natural}=G, \quad {\omega}^{\natural}=\theta, \quad {\theta}^{\natural}=\omega\,.
\]
Namely, the operation $\natural$ interchanges the roles of $\omega$ and $\theta$ and also $G$ and $G_{\ast}$. 

Kokubu, Umehara and Yamada gave a representation formula of flat fronts in ${\H}^{3}$ for a given 
pair of hyperbolic Gauss maps $(G, G_{\ast})$. 
\begin{theorem}[\cite{KUY1}, \cite{KUY2}]\label{thm-rep1}
Let $G$ and $G_{\ast}$ be nonconstant meromorphic functions on a Riemann surface $M$ such that $G(p)\not=G_{\ast}(p)$ 
for all $p\in M$. Assume that 
\begin{equation}\label{equ-period}
\displaystyle \int_{\gamma}\dfrac{dG}{G-G_{\ast}}\in i\R
\end{equation}
for every cycle $\gamma\in H_{1}(M,\Z)$. Set 
\begin{equation}\label{equ-invariant}
\xi (z)=c\cdot \exp{\displaystyle \int_{z_{0}}^{z} \dfrac{dG}{G-G_{\ast}}}
\end{equation}
where $z_{0}\in M$ is a reference point and $c\in \C\backslash \{0\}$ is an arbitrary constant. 
Then 
\begin{equation}\label{equ-L-lift}
\E=\left(
\begin{array}{cc}
G/\xi & \xi G_{\ast}/(G-G_{\ast}) \\
1/\xi & \xi/(G-G_{\ast})
\end{array}
\right)
\end{equation} 
is a nonconstant meromorphic Legendrian curve defined on $\widetilde{M}$ in $PSL(2,\C)$ 
whose hyperbolic Gauss maps are $G$ and $G_{\ast}$, and the projection $f=\E{\E}^{\ast}$ is single-valued on $M$. 
Moreover, $f$ is a front if and only if $G$ and $G_{\ast}$ have no common branch points. 
Conversely, any non-totally-umbilic flat front can be constructed this way. 
\end{theorem}
Throughout this paper, we call the condition \eqref{equ-period} the {\it period condition}. 
The canonical forms $\omega$, $\theta$ and the Hopf differential $Q$ of $f$ in Theorem \ref{thm-rep1} 
are given by  
\begin{equation}\label{equ-Wdata}
\omega = -\dfrac{1}{{\xi}^{2}}dG, \quad \theta =\dfrac{{\xi}^{2}}{(G-G_{\ast})^{2}}dG_{\ast}, 
\quad Q=-\dfrac{dGdG_{\ast}}{(G-G_{\ast})^{2}}\,.
\end{equation}
\red{It is clear} that there does not exist a flat front in ${\H}^{3}$ \red{both of} whose 
hyperbolic Gauss maps are constant. 

\begin{remark}
Kokubu, Umehara and Yamada obtained another construction of meromorphic Legendrian curves in $PSL(2,\C)$. 
For detail\red{s}, see \cite{KUY1}.  
\end{remark}

A front $f\colon M\to {\H}^{3}$ is said to be {\it complete} if there exists a symmetric $2$-tensor $T$ 
such that $T=0$ outside a compact set $C\subset M$ and $ds^{2}+T$ is a complete metric of $M$. 
In other words, the set of singular points of $f$ is compact and each divergent path has infinite length. 

\begin{theorem}[\cite{Hu}, \cite{GMM}, \cite{KUY2}] \label{thm-Huber}
Let $M$ be an oriented $2$-manifold and $f\colon M\to {\H}^{3}$ a complete flat front. 
Then $M$ is biholomorphic to ${\overline{M}}_{\gamma}\backslash \{p_{1},\ldots,p_{k}\}$, 
where ${\overline{M}}_{\gamma}$ is a closed Riemann surface of genus $\gamma$ and $p_{j}\in {\overline{M}}_{\gamma}$ 
$(j=1,\ldots,k)$. Moreover, the Hopf differential $Q$ of $f$ can be extended meromorphically to ${\overline{M}}_{\gamma}$. 
\end{theorem}

Each puncture point $p_{j}$ $(j=1,\cdots, k)$ is called  an {\it end} of $f$. G\'alvez, Mart\'inez and Mil\'an stud\red{ied} 
complete ends of flat surfaces in ${\H}^{3}$. 
The following fact is essentially prove\red{n} in \cite{GMM}. 

\begin{lemma}[\cite{GMM}, \cite{KUY2}] \label{le-ends}
Let $p$ be an end of \red{a} complete flat front. The following three conditions are equivalent\red{:}
\begin{enumerate}
\item The Hopf differential $Q$ has at most a pole of order $2$ at $p$. 
\item One hyperbolic Gauss map $G$ has at most a pole at $p$. 
\item The other hyperbolic Gauss map $G_{\ast}$ has at most a pole at $p$. 
\end{enumerate}
\end{lemma}

If an end of a flat front satisfies one of the three conditions above, it is called a {\it regular} end. 
An end that is not regular is called an {\it irregular} end. An end $p$ is said to be {\it embedded} 
if there exists a neighborhood $U$ of $p\in {\overline{M}}_{\gamma}$ such that the restriction 
of the front to $U\backslash \{p\}$ is an embedding. 

\begin{lemma}[\cite{KUY2}]\label{le-cplt}
The two hyperbolic Gauss maps take the same value at a regular end of a complete flat front, that is, 
$G(p)=G_{\ast}(p)$ if $p$ is a regular end. 
\end{lemma}

By the lemma above and investigation of embedded regular ends of complete flat fronts, 
Kokubu, Umehara and Yamada showed the following global properties of complete flat fronts. 

\begin{theorem}[\cite{KUY2}, Theorem 3.13] \label{KUY-Osserman}
Let $f\colon {\overline{M}}_{\gamma}\backslash \{p_{1}, \ldots, p_{k}\}\to {\H}^{3}$ be a complete flat front 
whose ends are all regular. Then 
\[
d+d_{\ast}\geq k\red{,}\, 
\]
where $d$ is the degree of $G$ considered \red{in} a map on ${\overline{M}}_{\gamma}$ $($if $G$ has essential singularities, 
then we define $d=\infty$$)$ and $d_{\ast}$ is the degree of $G_{\ast}$ considered as the same way. 
Furthermore, equality holds if and only if all ends are embedded. 
\end{theorem}

We remark that this inequality is \red{an} analogue of the Osserman inequality for algebraic minimal surfaces in ${\R}^{3}$ 
(\cite{Os1}, \cite{Os2}). 

\section{An effective estimate for the totally ramified value number of \red{the} hyperbolic Gauss maps}

We first recall the definition of the totally ramified value number of 
a meromorphic function on a Riemann surface. 

\begin{definition}[Nevanlinna \cite{Ne}]\label{trn}
Let $M$ be a Riemann surface and $h$ a meromorphic function on $M$. 
We call $b\in \C\cup \{\infty\}$ a {\it totally ramified value} of $h$ 
when at all the inverse image points of $b$, $h$ branches. 
We regard exceptional values also as totally ramified values. 
Let $\{a_{1},\ldots,a_{r_{0}},b_{1},\ldots,b_{l_{0}}\}\in \C\cup \{\infty\}$ be the set of all totally ramified values 
of $h$, where $a_{j}$ ($j=1,\ldots,r_{0}$) are exceptional values. 
For each $a_{j}$, \red{set} $\nu_{j}=\infty$, and for each $b_{j}$, 
define $\nu_{j}$ to be the minimum of the multiplicities of $h$ at points $h^{-1}(b_{j})$. 
Then we have $\nu_{j}\ge 2$. We call 
\[
\nu_{h}=\displaystyle\sum_{a_{j},b_{j}}\biggl(1-\frac{1}{\nu_{j}}\biggl)=r_{0}+\displaystyle\sum_{j=1}^{l_{0}}\biggl(1-\frac{1}{\nu_{j}}\biggl)
\]
the {\it totally ramified value number}  of $h$. 
\end{definition}

We next give an effective estimate for the totally ramified value number of the hyperbolic Gauss maps of 
complete flat fronts in ${\H}^{3}$. 

\begin{theorem}\label{thm-ramification}
Let $f\colon {\overline{M}}_{\gamma}\backslash \{p_{1}, \ldots, p_{k}\}\to {\H}^{3}$ be a complete flat front. 
If \red{the} two hyperbolic Gauss maps $G$ and $G_{\ast}$ are nonconstant and ${\nu}_{G}>2$ and ${\nu}_{G_{\ast}}>2$, then we have 
\begin{equation}\label{inequ:ramification}
\dfrac{1}{{\nu}_{G}-2}+\dfrac{1}{{\nu}_{G_{\ast}}-2}\geq \dfrac{k}{2\gamma -2+k}\,.
\end{equation}
\end{theorem}
Note that the right \red{hand} side of the inequality \eqref{inequ:ramification} \red{is} describe\red{d} in terms of only 
topological data \red{on} $M={\overline{M}}_{\gamma}\backslash \{p_{1}, \ldots, p_{k}\}$, that is, 
no data of the degrees of the hyperbolic Gauss maps \red{is used}. 

\begin{proof}
If $f$ has an irregular end, then $G$ or $G_{\ast}$ has an essential singularity there. By the big Picard theorem, 
we get ${\nu}_{G}\leq 2$ or ${\nu}_{G_{\ast}}\leq 2$. Thus we only \red{need to} consider the case where all ends are regular. 
Assume that $G$ is nonconstant and omits $r_{0}$ values. Let $d$ be the degree of $G$ considered as a map on 
$\overline{M}_{\gamma}$ and \red{let} $n_{0}$ be the sum of branching orders at the inverse image of these exceptional values of $G$. 
Then we have 
\begin{equation}\label{equ:ex-rami}
k\geq dr_{0} -n_{0}\,.
\end{equation}
Let $b_{1},\ldots,b_{l_{0}}$ be the totally ramified values which are not exceptional values. 
Let $n_{r}$ be the sum of branching order\red{s} at the inverse image of $b_{i}$ ($i=1,\ldots,l_{0}$) of $G$. 
For each $b_{i}$, we denote 
\[
\nu_{i}=\text{min}_{G^{-1}(b_{i})}\{\text{multiplicity of } G(z)=b_{i} \}, 
\] 
\red{and} then the number of points in the inverse image $G^{-1}(b_{i})$ is less than or equal to $d/{\nu}_{i}$. 
Thus we have
\begin{equation}\label{equ:TRV-rami}
\displaystyle dl_{0}-n_{r}\leq  \sum_{i=1}^{l_{0}} \dfrac{d}{{\nu}_{i}}\,.
\end{equation}
This implies 
\begin{equation}\label{equ-TRVN1}
\displaystyle l_{0}-\sum_{i=1}^{l_{0}} \dfrac{1}{{\nu}_{i}}\leq \dfrac{n_{r}}{d}\,.
\end{equation}
Let $n_{G}$ be the total branching order of $G$ on ${\overline{M}}_{\gamma}$. Then applying the Riemann-Hurwitz theorem 
to the meromorphic function $G$ on ${\overline{M}}_{\gamma}$, we obtain 
\begin{equation}\label{equ-RH}
n_{G}=2(d+\gamma -1)\,.
\end{equation}
Thus we get 
\begin{equation}\label{inequ-TRVN1}
\displaystyle {\nu}_{G}=r_{0}+\sum_{i=1}^{l_{0}}\biggl{(} 1- \dfrac{1}{{\nu}_{i}}\biggr{)}\leq 
\dfrac{n_{0}+k}{d}+\dfrac{n_{r}}{d}\leq \dfrac{n_{G}+k}{d}\leq 2+\dfrac{2\gamma -2+k}{d}\,. 
\end{equation}
Similarly, we get 
\begin{equation}\label{inequ-TRVN2}
{\nu}_{G_{\ast}}\leq 2+\dfrac{2\gamma -2+k}{d_{\ast}}\,. 
\end{equation} 
Here we assume that ${\nu}_{G}>2$ and ${\nu}_{G_{\ast}}>2$. Then we have 
\begin{equation}\label{inequ-TRVN3}
\dfrac{1}{{\nu}_{G}-2}\geq \dfrac{d}{2\gamma -2+k}, \quad \dfrac{1}{{\nu}_{G_{\ast}}-2}\geq \dfrac{d_{\ast}}{2\gamma -2+k}\,.
\end{equation} 
Combining these inequalities and Theorem \ref{KUY-Osserman}, we deduce \red{that}
\begin{equation}\label{inequ-TRVN4}
\dfrac{1}{{\nu}_{G}-2}+\dfrac{1}{{\nu}_{G_{\ast}}-2}\geq \dfrac{d+d_{\ast}}{2\gamma -2+k}\geq \dfrac{k}{2\gamma -2+k}\,.
\end{equation}
\end{proof}

As a corollary, we can get the upper bounds \red{for} the number of exceptional values of the hyperbolic Gauss maps 
of complete flat fronts in ${\H}^{3}$ \red{in} some topological cases. Here, we denote by $D_{G}$ and $D_{G_{\ast}}$ the number of 
exceptional values of $G$ and $G_{\ast}$, respectively. 

\begin{corollary}\label{Cor:excep1}
For complete flat fronts in ${\H}^{3}$, we have the following: 
\begin{enumerate}
\item[(i)] \blue{There does not exist a complete flat front with $\gamma =0$, $p\geq 4$ and $q\geq 4$.}
\item[(ii)] \blue{There does not exist a complete flat front with $\gamma =1$, $p\geq 5$ and $q\geq 5$.}
\end{enumerate}
\end{corollary}

\begin{proof}
When $\gamma=0$, $D_{G}>2$ and $D_{G_{\ast}}>2$, from the inequality \eqref{inequ:ramification}, we have
\[
\dfrac{1}{D_{G}-2}+\dfrac{1}{D_{G_{\ast}}-2}\geq \dfrac{k}{k-2}>1\,.
\]
On the other hand, if $\gamma=0$, $D_{G}\geq 4$ and $D_{G_{\ast}}\geq 4$, then it holds that 
\[
\dfrac{1}{D_{G}-2}+\dfrac{1}{D_{G_{\ast}}-2}\leq 1\,.
\]
Therefore, if $\gamma=0$, $D_{G}\geq 4$ and $D_{G_{\ast}}\geq 4$, then 
both $G$ and $G_{\ast}$ are constant\red{, so} 
 there does not exist such a front.  Hence we obtain (i). 
In the same way, when $\gamma=1$, $D_{G}>2$ and $D_{G_{\ast}}>2$, we have 
\[
\dfrac{1}{D_{G}-2}+\dfrac{1}{D_{G_{\ast}}-2}\geq 1\,.
\]
On the other hand, if $\gamma=1$, $D_{G}\geq 5$ and $D_{G_{\ast}}\geq 5$, then we get 
\[
\dfrac{1}{D_{G}-2}+\dfrac{1}{D_{G_{\ast}}-2}< 1\,.
\]
Therefore we obtain (ii). 
\end{proof}

Finally, we consider the Fujimoro-Hoffman-Osserman problem, that is, 
the problem of finding the common maximal number of the exceptional values of 
two hyperbolic Gauss maps of complete flat fronts in ${\H}^{3}$. 
We remark that the common maximal number of 
the exceptional values of the Gauss maps $g_{1}$ and $g_{2}$ of 
nonflat complete minimal surfaces in ${\R}^{4}$ is ``$4$'', 
that is, $D_{g_{1}}=D_{g_{2}}=4$ (\cite{Fu}, \cite{HO} and \cite{Ka2}). 
By Corollary \ref{Cor:excep1}, if $\gamma=0$, then the common maximal number of exceptional values of 
two hyperbolic Gauss maps is ``$3$'', that is, $D_{G}=D_{G_{\ast}}=3$. 
Moreover, if $\gamma=1$, then the common maximal number of exceptional values of two hyperbolic Gauss maps 
is ``$4$'', that is, $D_{G}=D_{G_{\ast}}=4$. Then we get necessary conditions for the existence of 
complete flat fronts whose hyperbolic Gauss maps have the common maximal number of exceptional values. 

\begin{corollary}
Let $f\colon \overline{M}_{\gamma}\backslash \{p_{1},\ldots, p_{k}\}\to {\H}^{3}$ be a complete flat front. 
\begin{enumerate}
\item[(i)] If ${\gamma}=0$ and $D_{G}=D_{G_{\ast}}=3$, then $k\geq 4$\,.
\item[(ii)] If ${\gamma}=1$ and $D_{G}=D_{G_{\ast}}=4$, then all ends are regular and embedded. 
\end{enumerate}
\end{corollary}

\begin{proof}
When ${\gamma}=0$, by the inequality \eqref{inequ:ramification}, we have 
\[
\dfrac{1}{D_{G}-2}+\dfrac{1}{D_{G_{\ast}}-2}\geq \dfrac{k}{k-2}\,.
\]
Moreover, if $D_{G}=3$ and $D_{G_{\ast}}=3$, then we have
\[
\dfrac{1}{D_{G}-2}+\dfrac{1}{D_{G_{\ast}}-2}=2\,.
\]
Therefore, for this case, we get the following inequality\red{:} 
\[
\dfrac{k}{k-2}\leq 2\,.
\]
Thus we obtain (i). Next we prove (ii). When $\gamma=1$, by \eqref{inequ-TRVN4}, we get 
\[
\dfrac{1}{D_{G}-2}+\dfrac{1}{D_{G_{\ast}}-2}\geq \dfrac{d+d_{\ast}}{k}\geq 1\,.
\]
Moreover, if $D_{G}=4$ and $D_{G_{\ast}}=4$, then we have
\[
\dfrac{1}{D_{G}-2}+\dfrac{1}{D_{G_{\ast}}-2}=1\,.
\]
Therefore, we can get the following equality\red{:} 
\[
d+d_{\ast}=k\,\red{.}
\]
By virtue of Theorem \ref{KUY-Osserman}, all ends are regular and embedded \red{in} this case. 
\end{proof}

\section{Examples of complete flat fronts from the \red{viewpoint} of value distribution of the hyperbolic Gauss maps}
In the first half of this section, we investigate examples of complete flat fronts in ${\H}^{3}$ from the \red{viewpoint} 
of \red{the} value distribution of the hyperbolic Gauss maps. 

\begin{example}[Example 4.1 of \cite{KUY2}]\label{ex-revolution}
We set ${\overline{M}}_{0}=\C\cup \{\infty \}$ and consider a pair $(G, G_{\ast})$ of meromorphic functions 
on ${\overline{M}}_{0}$ given by $G(z)=z$ and $G_{\ast}(z)={\alpha}z$, for some constant ${\alpha}\in \R\backslash \{1\}$. 
We define $M$ by $M={\overline{M}}_{0}\backslash \{0\}$ for the case where ${\alpha}=0$ 
and $M={\overline{M}}_{0}\backslash \{0, \infty\}$ for the case where ${\alpha}\not=0$, respectively. 
By Theorem \ref{thm-rep1}, we can construct a flat front $f\colon M\to {\H}^{3}$ whose hyperbolic Gauss maps 
are $G$ and $G_{\ast}$. Indeed we can easily see that $M$ and $(G, G_{\ast})$ satisfy the period condition and 
these data give a Legendrian immersion ${\E}_{f}$ of $f$
\begin{equation}\label{ex-Legendre1}
{\E}_{f}=\left(
\begin{array}{cc}
\dfrac{z^{-{\alpha}/(1-{\alpha})}}{c} & \dfrac{c{\alpha}z^{1/(1-\alpha)}}{1-\alpha} \smallskip \\
\dfrac{z^{-1/(1-{\alpha})}}{c} & \dfrac{cz^{{\alpha}/(1-\alpha)}}{1-\alpha}
\end{array}
\right) 
\; \text{for some constant}\: c\,.
\end{equation}
Moreover, the canonical forms $\omega$ and $\theta$ and the Hopf differential $Q$ of $f$ is given by 
\[
{\omega}=-\dfrac{1}{c^{2}}z^{-2/(1-\alpha)}dz, \quad {\theta}=\dfrac{c^{2}\alpha}{(1-{\alpha})^{2}}z^{2{\alpha}/(1-\alpha)}dz, 
\quad Q=-\dfrac{\alpha}{(1-{\alpha})^{2}}z^{-2}dz^{2}\,.
\]
Thus $f$ is complete. For the case where $\alpha\not =0$, the hyperbolic Gauss maps $G$ and $G_{\ast}$ of $f$ 
have the same exceptional values $0$ and $\infty$, that is, $D_{G}=D_{G_{\ast}}=2$. 
For the case where $\alpha=0$, $G$ has one exceptional value $0$ and $G_{\ast}$ is constant. Note that 
$f$ is a horosphere if $\alpha =0$. 
\end{example}

We remark that horospheres can be characterized by the hyperbolic Gauss maps as follows: 

\begin{theorem}[Proposition 4.2 of \cite{KUY2}]\label{thm-horosphere}
If one of \red{the} two hyperbolic Gauss maps of a complete flat front in ${\H}^{3}$ is constant, 
then it is a horosphere. 
\end{theorem}

\red{We have not found} a complete flat front whose two hyperbolic Gauss maps have the common maximal number 
of exceptional value\red{,} for both $\gamma=0$ and $\gamma=1$. 
However, there exists a complete flat front of genus $0$ with $(D_{G}, D_{G_{\ast}})=(3, 2)$. 

\begin{example}[Theorem 4.4 (iii) of \cite{KUY2}]
There exists a complete flat front $f\colon M= \C\backslash \{0, 1\}\to {\H}^{3}$ whose hyperbolic Gauss maps 
are 
\begin{equation}\label{example-ruy2}
(G, G_{\ast})=(z, z^{2}). 
\end{equation}
In particular, $D_{G}=3$ and $D_{G_{\ast}}=2$ and all ends are regular and embedded. 
\end{example}

In the latter half of this section, we give some new examples of complete flat fronts in ${\H}^{3}$. 
We first give an example of genus $0$ with $4$ regular embedded ends and $({\nu}_{G}, {\nu}_{G_{\ast}})=(3, 2)$. 

\begin{proposition}\label{prop-exK1}
There exists a complete flat front $f\colon M= \C\backslash \{0, \pm 1\}\to {\H}^{3}$ whose hyperbolic Gauss maps are 
\begin{equation}\label{example-K1}
(G, G_{\ast})=\biggl{(}z^{2}, \dfrac{z(z+a)}{az+1} \biggr{)} \quad (a\in \R\backslash \{0, \pm 1\})\,. 
\end{equation}
In particular, ${\nu}_{G}=3$ and ${\nu}_{G_{\ast}}=2$ and all ends are regular and embedded.  
\end{proposition}

\begin{proof} 
By \red{a} straightforward computation, we see that 
\[
\dfrac{dG}{G-G_{\ast}}=\dfrac{2(az+1)}{a(z+1)(z-1)}dz\,,
\]
and it is holomorphic at $z=0$ and has poles only at $z=\pm 1, \infty$. All of them are simple poles, 
with residues $(1+a)/a$, $(a-1)/a$, $-2$, respectively. By the condition \red{on} $a$, these residues are real. 
Thus these data satisfy the period condition. 
Moreover, we can clearly see that $G$ and $G_{\ast}$ take the same values at $z=0, \pm 1, \infty$ 
and have no common branch points. 
By Theorem \ref{thm-rep1}, we can construct a flat front $f\colon M\to {\H}^{3}$ 
whose hyperbolic Gauss maps are \eqref{example-K1}. 

On the other hand, the canonical forms $\omega$ and $\theta$ of $f$ are given by
\[
\omega= -\dfrac{2}{c^{2}}z(z+1)^{-2(a-1)/a}(z-1)^{-2(a+1)/a}dz, \quad 
\theta= \dfrac{c^{2}}{a^{2}}z^{-2}(z+1)^{-2/a}(z-1)^{2/a}(az^{2}+2z+a)dz\,.
\]
Furthermore, the Hopf differential of $f$ is given by 
\[
Q=-\dfrac{2(az^{2}+2z+a)}{a^{2}z(z+1)^{2}(z-1)^{2}}dz^{2}\,.
\]
Thus $Q$ has poles only at $z=0, \pm 1, \infty$ with 
\[
({\ord}_{0}Q,\,  {\ord}_{1}Q,\,  {\ord}_{-1}Q,\,  {\ord}_{\infty}Q)=(-1, -2, -2, -1). 
\]
Hence $f$ is complete.  

All ends of $f$ are regular and embedded because $f$ satisfies equality \red{in the equation in} Theorem \ref{KUY-Osserman}. 
One hyperbolic Gauss map $G$ has three exceptional values $0$, $1$, $\infty$. 
The other hyperbolic Gauss map $G_{\ast}$ has one exceptional value $0$ and two totally ramified values. 
Therefore, we \red{see} that ${\nu}_{G}=1+1+1=3$ and ${\nu}_{G_{\ast}}=1+(1/2)+(1/2)=2$.  
\end{proof}

\begin{remark}
By virtue of Theorem \ref{KUY-Osserman}, if a complete flat front has $4$ embedded regular ends, 
then $(d, d_{\ast})=(1, 3),$
$(2, 2)$ or $(3, 1)$. By Example 4.5 of \cite{KUY2} and Proposition \ref{prop-exK1}, we see that 
there exists \red{an} example for \red{any of the} cases \red{where} $d+d_{\ast}=4$. 
\end{remark}

We next give an example of \red{a} complete flat front of genus $0$ with $(d,d_{\ast})=(3, 2)$ and $5$ regular embedded ends. 

\begin{proposition}\label{prop-exK2}
There exists a complete flat front $f\colon M= \C\backslash \{0, 1, -2, -3/2\}\to {\H}^{3}$ 
whose hyperbolic Gauss maps are  
\begin{equation}\label{example-K2}
(G, G_{\ast})=\biggl{(}z^{3}, \dfrac{z(z+6)}{2z+5} \biggr{)} . 
\end{equation}
In particular, ${\nu}_{G}=3$ and ${\nu}_{G_{\ast}}=1$ and all ends are regular and embedded.  
\end{proposition}

\begin{proof} 
By \red{a} straightforward computation, we see that 
\[
\dfrac{dG}{G-G_{\ast}}=\dfrac{3z(2z+5)}{(z-1)(z+2)(2z+3)}dz\,,
\]
and \red{this} is holomorphic at $z=0$ and has poles only at $z=1, -2, -3/2, \infty$. All of them are simple poles, 
with residues $7/5$, $-2$, $18/5$, $-3$, respectively. 
Thus these data satisfy the period condition. 
Moreover, we can easily see that $G$ and $G_{\ast}$ take the same values at $z=0, 1, -2, -3/2, \infty$ and 
have no common branch points. 
By Theorem \ref{thm-rep1}, we can construct a flat front $f\colon M\to {\H}^{3}$ 
whose hyperbolic Gauss maps are \red{as in} \eqref{example-K2}. 

On the other hand,  the canonical forms $\omega$ and $\theta$ of $f$ are given by
\begin{eqnarray*}
\omega &=& -\dfrac{3}{c^{2}}z^{2}(z-1)^{-14/5}(z+2)^{4}(2z+3)^{-36/5}dz\, ,  \\
\theta &=& 2c^{2}z^{-2}(z-1)^{4/5}(z+2)^{-6}(2z+3)^{26/5}(z^{2}+6z+15)dz\,.
\end{eqnarray*}
Furthermore, the Hopf differential of $f$ is given by 
\[
Q=-\dfrac{6(z^{2}+6z+15)}{(z-1)^{2}(z+2)^{2}(2z+3)^{2}}dz^{2}\,.
\]
Thus $Q$ has poles only at $z=1, -2, -3/2$ with 
\[
({\ord}_{1}Q,\,  {\ord}_{-2}Q,\,  {\ord}_{-3/2}Q)=(-2, -2, -2). 
\]
Hence $f$ is complete.  

All ends of $f$ are regular and embedded because $f$ satisfies equality \red{in the equation in} 
Theorem \ref{KUY-Osserman}. 
One hyperbolic Gauss map $G$ has two exceptional values $0$, $\infty$. 
The other hyperbolic Gauss map $G_{\ast}$ has two totally ramified value\red{s}. 
Therefore, we \red{see} that ${\nu}_{G}=2$ and ${\nu}_{G_{\ast}}=(1/2)+(1/2)=1$.  
\end{proof}

Finally, we give an example of \red{a} complete flat front in ${\H}^{3}$ of genus $1$ with $5$ regular ends. 
Let ${\overline{M}}_{1}$ be the square torus on which the Weierstrass $\wp$ function satisfies 
\[
({\wp}')^{2}=4{\wp}({\wp}^{2}-a^{2}), \quad a={\wp}(1/2)\,.
\]
\begin{proposition}\label{prop-exK3}
There exists a complete flat front $f\colon{\overline{M}}_{1}\backslash \{z;\, {\wp}(z)({{\wp}(z)}^{2}+a^{2})=0\} 
\to {\H}^{3}$ whose hyperbolic Gauss maps are 
\begin{equation}\label{example-K3}
(G, G_{\ast})=\biggl{(}\dfrac{{\wp}'}{\wp}, \dfrac{2({\wp}^{2}-3a^{2})}{{\wp}'}\biggr{)}\red{,}
\end{equation}
\red{with} $5$ regular ends. 
\end{proposition}

\begin{proof}
For this data, a computation gives 
\[
\dfrac{dG}{G-G_{\ast}}=d\log{\wp}\,.
\]
This implies that these data satisfy the period condition. 
Moreover, $G$ and $G_{\ast}$ take the same values on $\{z;\, {\wp}(z)({{\wp}(z)}^{2}+a^{2})=0\}$ and 
have no common branch points. By Theorem \ref{thm-rep1}, we can construct a flat front $f\colon M\to {\H}^{3}$ 
whose hyperbolic Gauss maps are \eqref{example-K3}. 

The canonical forms $\omega$, $\theta$ and the Hopf differential $Q$ of $f$ are given by
\[
\omega= -\dfrac{2({\wp}^{2}+a^{2})}{c^{2}{\wp}^{3}}dz, \quad 
\theta= \dfrac{c^{2}{\wp}^{2}({\wp}^{4}+6a^{2}{\wp}^{2}-3a^{4})}{({\wp}^{2}+a^{2})^{2}}dz, \quad
Q=\dfrac{2({\wp}^{4}+6a^{2}{\wp}^{2}-3a^{4})}{{\wp}({\wp}^{2}+a^{2})}dz^{2}
\]
from which the completeness of the ends $\{z;\, {\wp}(z)({{\wp}(z)}^{2}+a^{2})=0\}$ follows. 
Obviously all ends are regular but not embedded because $f$ does not satisfy equality \red{in the equation in} 
Theorem \ref{KUY-Osserman}. Indeed, we clearly see  that $d=2$ and $d_{\ast}=4$ and $6=d+d_{\ast}>k=5$. 
\end{proof}

\begin{remark}
There exists a complete flat front of genus $1$ with $(d, d_{\ast})=(3, 2)$ and $5$ 
regular embedded ends \cite[Example 4.6]{KUY2}. 
\end{remark}

\end{document}